\documentclass[a4paper,12pt]{amsart}
\usepackage{verbatim}
\usepackage{mathtools}
\usepackage{amsmath}
\usepackage{enumitem}
\usepackage{cite}
\setenumerate[1]{label=(\alph{*}), ref=(\alph{*})}
\setenumerate[2]{label=(\roman{*}), ref=(\roman{*})}

\newtheorem{theorem}{Theorem}[section]
\newtheorem{lemma}[theorem]{Lemma}
\newtheorem{proposition}[theorem]{Proposition}
\newtheorem{corollary}[theorem]{Corollary}
\newtheorem{ques}[theorem]{Question}

\theoremstyle{definition}
\newtheorem{definition}[theorem]{Definition}
\newtheorem{example}[theorem]{Example}

\theoremstyle{remark}
\newtheorem{remark}[theorem]{Remark}

\numberwithin{equation}{section}

\def\N{{\mathbb N}}
\def\R{{\mathbb R}}

\newcommand{\setN}{\{1,\ldots,n\}}

\DeclareMathOperator{\linspan}{span}
\DeclareMathOperator{\spann}{span}

\DeclareMathOperator{\supp}{supp}
\newcommand{\HB}{\text{H{\kern -0.35em}B}}

\newcommand{\Lip}{\textnormal{Lip}}

\newcommand{\spanof}{{\rm
span} \ }

\newcommand{\xast}{x^{\ast}}

\newcommand{\yast}{y^{\ast}}

\newcommand{\Xast}{X^{\ast}}

\newcommand{\Xastast}{X^{\ast\ast}}

\newcommand{\n}[1]{\left\|#1\right\|}

\begin{document}

\title{Symmetric strong diameter two property}
\date{}

\author{Rainis Haller, Johann Langemets, Vegard Lima, and Rihhard Nadel}
\address{Institute of Mathematics, University of Tartu, J.~Liivi 2, 50409 Tartu, Estonia}
\curraddr{}
\email{rainis.haller@ut.ee, johann.langemets@ut.ee, nadel@ut.ee}

\address[V.~Lima] {NTNU, Norwegian University of Science and
  Technology, Aalesund, Postboks 1517, N-6025 {\AA}lesund Norway.}
\curraddr{Department of Engineering Sciences, University of Agder,
Postboks 422, 4604 Kristiansand, Norway.}
\email{Vegard.Lima@uia.no}

\thanks{R.~Haller, J.~Langemets, and R.~Nadel were partially supported by institutional research funding IUT20-57 of the Estonian Ministry of Education and Research.}




\date{}

\dedicatory{}

\begin{abstract}
We study Banach spaces with the property that, given
a finite number of slices of the unit ball,
there exists a direction such that
all these slices contain a line segment of
length almost 2 in this direction.
This property was recently named the
\emph{symmetric strong diameter two property}
by Abrahamsen, Nygaard, and P{\~o}ldvere.

The symmetric strong diameter two property
is not just formally stronger than
the strong diameter two property
(finite convex combinations of slices have
diameter 2).
We show that the symmetric strong diameter
two property is only preserved by $\ell_\infty$-sums,
and working with weak star slices we show that
$\Lip_0(M)$ have the weak star version of
the property for several classes of metric spaces $M$.
\end{abstract}

\maketitle

\section{Introduction}\label{sec: Introduction}

All Banach spaces considered in this paper are nontrivial and over
the real field. The closed unit ball of a Banach space $X$ is denoted
by $B_X$ and its unit sphere by $S_X$. The dual space of $X$ is
denoted by $X^\ast$ and the bidual by $X^{\ast\ast}$.
By a \emph{slice} of $B_X$ we mean a set of the form
\begin{equation*}
  S(B_X, x^*,\alpha) :=
  \{
  x \in B_X : x^*(x) > 1 - \alpha
  \},
\end{equation*}
where $x^* \in S_{X^*}$ and $\alpha > 0$.
If $X$ is a dual space, then slices
whose defining functional comes from (the canonical image of) the
predual of $X$ are called weak$^\ast$ slices.

This research belongs to the area of diameter 2 properties, which is a recent topic in geometry of Banach spaces and has received intensive attention in the last years (see \cite{{ALL},ALMN,ALN, ALN2, ANPssd2p,ABGLP2012,BLR_octahedralityLipschitz,MR3466077, becerra_guerrero_2016,HL1,HLN,haller_duality_2015,Ivakhno,PR_CharacterizationOH}). Its central research objects are Banach spaces where certain subsets of the unit ball (slices, relatively weakly open subsets or convex combinations of slices) all have diameter equal to 2. Different subsets under consideration led to particular diameter properties. 

Recall from \cite{ALL} that a Banach space $X$
is \emph{almost square} (ASQ) if whenever $n \in \mathbb{N}$,
and $x_1,\dotsc,x_n \in S_X$, there exists a sequence
$(y_k)_{k=1}^\infty \subset S_X$ such that
$\|x_i \pm y_k\| \to_k 1$ for every $i \in \setN$. If a Banach space $X$ is ASQ, then every finite convex combination
of slices of $B_X$ has diameter two \cite[Proposition~2.5]{ALL}, that is, $X$ has the \emph{strong diameter 2 property} (SD2P).  Spaces which are ASQ include $c_0(X_n)$,
where $X_n$ are arbitrary Banach spaces,
and Banach spaces $X$ which are M-ideals
in $X^{**}$ (see \cite{ALL}).

In this paper we investigate the following property, which first appeared in \cite{ALN},
but was singled out and studied in \cite{ANPssd2p}.
\begin{definition}
  A Banach space $X$ has the
  \emph{symmetric strong diameter 2 property} (SSD2P) if for every
  finite family $\{S_i\}_{i=1}^n$ of slices of $B_X$ and
  $\varepsilon > 0$, there exist $x_i \in S_i$ and $y \in B_X$,
  independent of $i$, such that $x_i \pm y \in S_i$ for every
  $i \in \setN$ and $\|y\| > 1 - \varepsilon$.
\end{definition}

It is known \cite[Lemma~4.1]{ALN} that if a Banach
space has the SSD2P, then it has the
SD2P. In fact, the SSD2P is strictly stronger than
the SD2P. For example, $L_1[0,1]$ has the SD2P,
but not the SSD2P
(see Remark~\ref{rem:L1-not-SSD2P} below). On the other hand, ASQ Banach spaces have the SSD2P (this can easily be observed from Theorem~\ref{ssd2p:equiv_forms}~\ref{item:ssd2p-char-d} below). The converse fails, $C[0,1]$ has the SSD2P (see Example~\ref{ex: C[0,1] SSD2P}) and is not ASQ (this can be easily seen by considering the constant 1 function).


The following classes of spaces have the SSD2P:
  \begin{enumerate}
  \item Lindenstrauss spaces (this follows by inspecting the proof of Proposition~4.6 in \cite{ALN2});
  \item uniform algebras (see Theorem~4.2 in \cite{ALN});
  \item ASQ-spaces, in particular, Banach spaces which are M-ideals in their bidual (see \cite{ALL});
  \item Banach spaces with an infinite-dimensional centralizer (this follows by inspecting the proof of Proposition~3.3 in \cite{ABGLP2012});
  \item somewhat regular linear subspaces of $C_0(L)$, whenever $L$ is an infinite locally compact Hausdorff topological space  \cite{ANPssd2p};
  \item M\"untz spaces (this follows by inspecting the proof of Theorem~2.5 in \cite{ALMN}).
  \end{enumerate}
All of the above-listed spaces contain an almost isometric copy of $c_0$. However, we do not know whether every space with the SSD2P contains $c_0$. On the other hand, every Banach space containing a copy of $c_0$ can
be equivalently renormed to have the SSD2P, in fact even to be ASQ (see \cite{MR3466077}).

Let us summarize the results of the paper.
We start our investigation in Section~\ref{sec: Characterization} by giving equivalent formulations of the SSD2P, which are often more convenient to use. 
Recently in \cite{HLN}, it was proven that the SD2P is preserved by a lot of absolute normalized norms. However, in Section~\ref{sec: Direct sums}, we show that the only direct sums of Banach spaces that can have the SSD2P are the $\ell_\infty$-sums. In Section~\ref{sec: Subspaces}, we prove that the SSD2P passes down from a superspace if a subspace is an ai-ideal in it or if the quotient is strongly regular. For the other way, we show that if a subspace is an M-ideal in the superspace then the SSD2P will lift from the subspace to the superspace. In Section~\ref{sec: Lipschitz}, we study the SSD2P for Lipschitz spaces. At the end we list some open problems.




\section{Characterization of the SSD2P}\label{sec: Characterization}

Let $\mathcal{O}(x)$ denote
the set of all relatively weakly open neighborhoods
of $x$ in $B_X$.

\begin{theorem}\label{ssd2p:equiv_forms}
  Let $X$ be a Banach space. The following assertions are equivalent:
  \begin{enumerate}
  \item\label{item:ssd2p-char-a}
    $X$ has the SSD2P.
  \item\label{item:ssd2p-char-b}
    Whenever $n \in \mathbb{N}$, $U_1, \ldots, U_n$
    are relatively weakly open subsets of $B_X$
    and $\varepsilon > 0$,
    there exist $x_i \in U_i$, $i \in \setN$,
    and $y \in B_X$ such that $x_i \pm y \in U_i$ for
    every $i \in \setN$ and $\|y\| > 1 - \varepsilon$.
  \item\label{item:ssd2p-char-c}
    Whenever $n \in \mathbb{N}$, $C_1, \ldots, C_n$
    are finite convex combinations of slices of $B_X$
    and $\varepsilon > 0$,
    there exist $x_i \in C_i$, $i \in \setN$,
    and $y \in B_X$ such that $x_i \pm y \in C_i$ for
    every $i \in \setN$ and $\|y\| > 1 - \varepsilon$.
  \item\label{item:ssd2p-char-d}
    Whenever $n \in \mathbb{N}$,
    $x_1, \dots, x_n \in S_X$,
    there exist nets
    $(y^i_\alpha) \subset S_X$ and
    $(z_\alpha) \subset S_X$ such that
    $y^i_\alpha \to_\alpha x_i$ weakly, $z_\alpha \to_\alpha 0$
    weakly, and $\|y^i_\alpha \pm z_\alpha\|\rightarrow 1$
    for all $i \in \setN$.
  \item\label{item:ssd2p-char-e}
    Whenever $n \in \mathbb{N}$,
    $x_1, \ldots, x_n \in S_X$,
    $U_i \in \mathcal{O}(x_i)$, $i \in \setN$,
    $V \in \mathcal{O}(0)$, and $\varepsilon > 0$,
    there exist $y_i \in U_i \cap S_X$, $i \in \setN$,
    and $z \in V \cap S_X$ such that
    $\|y_i \pm z\| \leq 1 + \varepsilon$.
  \end{enumerate}
\end{theorem}

\begin{proof}
  \ref{item:ssd2p-char-a} $\Rightarrow$ \ref{item:ssd2p-char-b}.
  Let $n \in \mathbb{N}$ and assume that $U_1, \ldots, U_n$
  are relatively weakly open subsets of $B_X$
  and that $\varepsilon > 0$.
  By Bourgain's lemma \cite[Lemma~II.1]{GGMS}
  each $U_i$ contains a convex combination of slices,
  say $U_i \supset \sum_{j=1}^{n_i} \lambda_i^j S_i^j$,
  with $\sum_{j=1}^{n_i} \lambda_i^j = 1$ and $\lambda_i^j > 0$,
  for each $i \in \setN$.
We apply the definition of the SSD2P to the family of all $S_i^j$ to find $x_i^j \in S_i^j$ and $y \in B_X$
  such that $x_i^j \pm y \in S_i^j$ and
  $\|y\| > 1 - \varepsilon$.
Set $w_i := \sum_{j=1}^{n_i} \lambda_i^j x_i^j$. Then
  \begin{equation*}
    w_i \in \sum_{j=1}^{n_i} \lambda_i^j S_i^j \subset U_i
  \end{equation*}
and
  \begin{equation*}
    w_i \pm y = \sum_{j=1}^{n_i} \lambda_i^j (x_i^j \pm y)
    \in \sum_{j=1}^{n_i} \lambda_i^j S_i^j \subset U_i.
  \end{equation*}
 This shows \ref{item:ssd2p-char-a} $\Rightarrow$ \ref{item:ssd2p-char-b}. The same proof also gives
  \ref{item:ssd2p-char-a} $\Rightarrow$ \ref{item:ssd2p-char-c},
  while
  \ref{item:ssd2p-char-b} $\Rightarrow$ \ref{item:ssd2p-char-a}
  and
  \ref{item:ssd2p-char-c} $\Rightarrow$ \ref{item:ssd2p-char-a}
  are trivial.

  \ref{item:ssd2p-char-b} $\Rightarrow$ \ref{item:ssd2p-char-e}.
  Let $n \in \mathbb{N}$, $x_1, \dots, x_n \in S_X$,
  $U_i \in \mathcal{O}(x_i)$, $i \in \setN$, $V \in \mathcal{O}(0)$,
  and $\varepsilon\in (0,1)$.
  By choosing $0 < \delta < \varepsilon$ small enough
  there exist finite sets
  $A_i \subset S_{X^*}$ and $B \subset S_{X^*}$
  such that
  \begin{equation*}
    U_i \supset \bar{U}_i := \{
    x \in B_X : |x^*(x - x_i)| < \delta, x^* \in A_i
    \}
  \end{equation*}
  and
  \begin{equation*}
    V \supset \bar{V} := \{
    x \in B_X : |x^*(x)| < \delta, x^* \in B
    \}.
  \end{equation*}
  Let
  \begin{equation*}
    \tilde{U}_i := \{
    x \in B_X : |x^*(x - x_i)| < \delta/2, x^* \in A_i
    \} \subset \bar{U}_i
  \end{equation*}
  and
  \begin{equation*}
    \tilde{V} := \{
    x \in B_X : |x^*(x)| < \delta/2, x^* \in B
    \} \subset \bar{V}.
  \end{equation*}
  For each $i \in \setN$,
  choose $x_i^* \in S_{X^*}$ such that $x_i^*(x_i) = 1$,
  and define $S_i := S(B_X, x_i^*,\delta/2)$.
  We apply \ref{item:ssd2p-char-b} to the
  relatively weakly open sets
  $W_i = S_i \cap \tilde{U}_i$ and $\tilde{V}$
  and find $w_i \in W_i$ and $v \in \tilde{V}$
  and $z \in B_X$ such that 
  \[
  w_i \pm z \in W_i,\quad
  v \pm z \in \tilde{V},\quad \text{ and }\quad \|z\| > 1 - \frac{\delta}2.
  \]

  Define $u_i := \frac{w_i}{\|w_i\|}$.
  Since $w_i \in S_i$ we get
  \[
 \|w_i\| > 1 - \frac{\delta}2\quad
  \text{ and }\quad \|u_i - w_i\| < \frac{\delta}2.
  \]
  From this and $w_i \in \tilde{U}_i$ we have $u_i \in U_i$.

  Next, note that $-(v \pm z) \in \tilde{V}$
  hence $z = \frac{1}{2}(-v + z) + \frac{1}{2}(v + z) \in
  \tilde{V}$ by convexity.
  Since $\|z\| > 1 - \delta/2$ we get that
  $y := \frac{z}{\|z\|} \in V$.

  Finally, note that
  \begin{equation*}
    \|u_i \pm y\| \le \|w_i - u_i\|
    + \|w_i \pm z\| + \|z - y\|
    \le \frac{\delta}{2} + 1 + \frac{\delta}{2}
    < 1 + \varepsilon.
  \end{equation*}

  \ref{item:ssd2p-char-d} $\Rightarrow$ \ref{item:ssd2p-char-a}. Let $n\in\N$,  $S_1:= S(B_X,\xast_1,\alpha_1),\dots, S_n:=S(B_X,\xast_n,\alpha_n)$ be slices of $B_X$ and $\varepsilon\in (0,1)$. Find a $\delta>0$ such that 
  \[
 \frac1{1+\delta}>1-\varepsilon\quad \text{ and }\quad  \frac{1-2\delta}{1+\delta}>1-\alpha_i
  \]
  for every $i\in\setN$.  For every $i$ choose an $x_i\in S_X\cap S(B_X,\xast_i,\delta)$. By \ref{item:ssd2p-char-d} there are nets $(y^i_\alpha) \subset S_X$ and
    $(z_\alpha) \subset S_X$ such that
    $y^i_\alpha \to_\alpha x_i$ weakly, $z_\alpha \to_\alpha 0$
    weakly, and $\|y^i_\alpha \pm z_\alpha\|\rightarrow 1$
    for all $i \in \setN$. 
    
    Find an index $\alpha_0$ such that 
    \[
    y^i_{\alpha_0}\in S(B_X,\xast_i,\delta),\quad  \|\xast_i(z_{\alpha_0})\|<\delta,\quad \text{ and }\quad \max \|y^i_{\alpha_0}\pm z_{\alpha_0}\|\leq 1+\delta
    \]
    for every $i\in \setN$.
    Finally, set $y_i:=y^i_{\alpha_0}/(1+\delta)$ and $z:=z_{\alpha_0}/(1+\delta)$. Then we have $y_i, y_i\pm z\in S_i$ and $z\in B_X$ with $\|z\|>1-\varepsilon$.
    
    The implication
  \ref{item:ssd2p-char-e} $\Rightarrow$ \ref{item:ssd2p-char-d}
  is straightforward.
\end{proof}

From Theorem~\ref{ssd2p:equiv_forms}~\ref{item:ssd2p-char-d} one can quickly verify that $C[0,1]$ has the SSD2P.

\begin{example}\label{ex: C[0,1] SSD2P}
  $C[0,1]$ has the SSD2P.
  Let $(f_i)_{i=1}^n \subset S_{C[0,1]}$.
  Choose $a > 0$ with $2a n < 1$
  and points $(x_i)_{i=1}^n \subset [0,1]$
  such that $|f_i(x_i)| = 1$.

  The set $R = [0,1] \setminus \cup_{i=1}^n (x_i-a,x_i+a)$
  is nonempty and contains a sequence $(U_k)^{\infty}_{k=1}$ of
  pairwise disjoint subsets.

  Let $g_k \ge 0$ be such that $\supp g_k \subseteq U_k$
  and $\|g_k\| = 1$.
  Then $(g_k)^{\infty}_{k=1}$ is a bounded sequence that
  converges pointwise to $0$, hence $g_k \to 0$ weakly.

  Define $y_{i,k} = (1-g_k)f_i$ and $z_k = g_k$.
  Clearly $y_{i,k} \to f_i$ weakly and $\|y_{i,k}\| = 1$
  since $x_i \notin U_m$. Also $y_{i,k}$ and $g_k$
  have disjoint support so that $\|y_{i,k} \pm z_k\| = 1$.
\end{example}


\section{Direct sums with the SSD2P}\label{sec: Direct sums}


We recall that a norm $N$ on $\mathbb R^2$ is \emph{absolute} (see
\cite{bonsall_numerical_1973}) if
\begin{equation*}
  N(a,b) = N(|a|,|b|) \qquad \text{for all}
  \ (a,b) \in \mathbb{R}^2
\end{equation*}
and \emph{normalized} if
\begin{equation*}
  N(1,0) = N(0,1) = 1.
\end{equation*}
For $1 \leq p \leq \infty$, we denote the $\ell_p$-norm on
$\mathbb{R}^2$ by $\|\cdot\|_p$.
Every norm $\|\cdot\|_p$ is absolute and normalized.
Moreover, if $N$ is an absolute normalized norm on
$\mathbb{R}^2$
(see \cite[Lemmata~21.1 and 21.2]{bonsall_numerical_1973}), then
\begin{equation*}
  \|\cdot\|_\infty \leq N(\cdot) \leq \|\cdot\|_1
\end{equation*}
and if $(a,b),(c,d) \in \mathbb{R}^2$ and
\begin{equation*}
  |a|\leq |c|\quad \text{and}\quad |b|\leq |d|,
\end{equation*}
then
\begin{equation*}
  N(a,b)\leq N(c,d).
\end{equation*}


If $X$ and $Y$ are Banach spaces and $N$ is an absolute normalized
norm on $\R^2$, then we denote by $X\oplus_N Y$ the product space
$X\times Y$ with respect to the norm
\begin{equation*}
  \|(x,y)\|_N = N(\|x\|,\|y\|)
  \qquad
  \text{for all}\ x \in X
  \ \text{and}\ y \in Y.
\end{equation*}
In the special case where $N$ is the $\ell_p$-norm,
we write $X \oplus_p Y$.

We will now prove that the $\ell_\infty$-norm is the only absolute normalized norm, which preserves the SSD2P.

\begin{theorem}\label{thm: direct sums with SSD2P}
  Let $X$ and $Y$ be Banach spaces.
  \begin{itemize}
  \item[(a)] $X \oplus_\infty Y$ has the SSD2P if and only if $X$ or $Y$ has the SSD2P.
  \item[(b)] If $N$ is an absolute
  normalized norm different from the $\ell_\infty$-norm,
  then $X\oplus_N Y$ does not have the SSD2P.
  \end{itemize}
\end{theorem}

\begin{proof}
(a). Assume first that $X$ has the SSD2P and denote by $Z:=X \oplus_\infty Y$. For every $i \in \setN$ let $W_i$ be a nonempty relatively weakly open subset of $B_Z$ containing the element $(u_i,v_i)$, and $\varepsilon > 0$. Find nonempty relatively weakly open subsets $U_i\subset B_X$ and $V_i\subset B_Y$ such that
  \[
  (u_i,v_i)\in U_i\times V_i\subset W_i.  
  \]
Since $X$ has the SSD2P, by Theorem~\ref{ssd2p:equiv_forms}~\ref{item:ssd2p-char-b}, we can find $x_i\in U_i$ and $x\in B_X$ such that $x_i, x_i\pm x\in U_i$ and $\|x\|>1-\varepsilon$. Set $z_i=(x_i,v_i)$ and $z=(x,0)$. Then $z_i, z_i\pm z\in W_i$ and $\|z\|>1-\varepsilon$, which completes the proof.

Assume now that $X \oplus_\infty Y$ has the SSD2P. Suppose for contradiction that $X$ and $Y$ both fail to have the SSD2P. 

Since $X$ fails the SSD2P, there are nonempty relatively weakly open subsets $U_1,\dots,U_n\in B_X$ and an $\varepsilon>0$ such that for all $x_i\in U_i$ and for all $x\in B_X$ with $\|x\|>1-\varepsilon$ there is an index $i_0$ such that $x_{i_0}+x\notin U_{i_0}$ or $x_{i_0}-x\notin U_{i_0}$. Also, there are nonempty relatively weakly open subsets $V_1,\dots,V_m\in B_Y$ and a $\delta>0$ such that for all $y_j\in V_j$ and for all $y\in B_Y$ with $\|y\|>1-\delta$ there is an index $j_0$ such that $y_{j_0}+y\notin V_{j_0}$ or $y_{j_0}-y\notin V_{j_0}$.

Set $W_{ij}:= U_i\times V_j$ for every $i\in\setN$ and $j\in\{1,\dots,m\}$. Then each $W_{ij}$ is a nonempty relatively open subset of $B_{X \oplus_\infty Y}$ and by our assumption there should be $(x_{i_0}, y_{j_0})\in W_{i_0j_0}$ and $(x,y)\in B_Z$ such that $(x_{i_0}, y_{j_0})\pm (x,y)\in W_{i_0j_0}$ and $\|(x, y)\|>1-\max\{\delta,\varepsilon\}$, which is impossible. 

(b). Denote $Z:=X \oplus_N Y$. Note that $N(1,1) > 1$, because $N$ differs from the $\ell_\infty$-norm. Let $a \in (0,1)$ be such that $N(a,a) = 1$.
  Since $N(a,1) > 1$ and $N(1,a) > 1$,
  there is a $\delta > 0$ such that if $N(u,v) \leq 1$
  and $u > 1- \delta$, then $v < a-\delta$ or if $v > 1- \delta$,
  then $u < a - \delta$.
  Fix an $\varepsilon > 0$ with
  $a - \delta \leq (1 - \varepsilon)a$.

  Consider slices $S_1 := S(B_{Z}, (\xast,0), \delta)$
  and $S_2 := S(B_{Z}, (0,\yast), \delta)$.
  Suppose for contradiction that $Z$ has the SSD2P,
  then there are $z_1 = (x_1,y_1) \in S_1$,
  $z_2 = (x_2,y_2) \in S_2$, and $w = (u,v) \in B_Z$ such that
  \begin{equation*}
    z_1 \pm w \in S_1, \quad z_2 \pm w \in S_2,
    \quad \text{ and }\quad \|w\| > 1 - \varepsilon. 
  \end{equation*}
  Therefore $(\xast,0)(z_1 \pm w) = \xast(x \pm u) > 1 - \delta$,
  which implies that $\|x_1 \pm u\| > 1 - \delta$.
  Similarly we have that $\|y_2 \pm v\| > 1 - \delta$.
  Hence $\|y_1 \pm v\| < a - \delta$ and
  $\|x_2 \pm u\| < a - \delta$.
  Now we see that
  \begin{equation*}
    \|v\| \leq \frac{1}{2}(\|y_1 + v\| + \|y_1 - v\|)
    < a - \delta.
  \end{equation*}
  Similarly, one has that $\|u\| < a - \delta$. Thus
  \begin{align*}
    1-\varepsilon &<\|w\|= N(\|u\|, \|v\|)\\
                   & \leq N(a - \delta, a - \delta) \\
                  &\leq N((1 - \varepsilon)a, (1 - \varepsilon)a)\\
                  &= (1 - \varepsilon) N(a,a)
                    = 1 - \varepsilon,
  \end{align*}
  a contradiction.
\end{proof}

\begin{remark}
 Note that Theorem~\ref{thm: direct sums with SSD2P} implies that almost squareness is also preserved only by $\ell_\infty$-sums.
\end{remark}

\begin{remark}\label{rem:L1-not-SSD2P}
  By Theorem~\ref{thm: direct sums with SSD2P}, $L_1[0,1]$ does not have
  the SSD2P, because $L_1[0,1] =
  L_1[0,\frac{1}{2}] \oplus_1 L_1[\frac{1}{2},1]$.
\end{remark}

If $(X_n)_{n=1}^\infty$ is a sequence of Banach spaces,
then $\ell_\infty(X_n)$ is the Banach space of
bounded sequences $(x_n)_{n=1}^\infty$, where
$x_n \in X_n$, with norm $\|(x_n)\| = \sup_{n}\|x_n\|$.

\begin{proposition}
  Let $(X_n)_{n=1}^\infty$ be a sequence of Banach spaces.
  Then $\ell_\infty(X_n)$ has the SSD2P.
\end{proposition}

\begin{proof}
  Define $Z := \ell_\infty(X_n)$ and $Z_0:= c_0(X_n)$.
  Let $P_k\colon Z\rightarrow Z_0$, $z=(x_n)\mapsto P_k(z)=(0, \ldots, 0, x_k, 0, \ldots)$. Observe that $(P_k(z))_{k=1}^\infty$ is a weakly null sequence in $Z$, because it is a weakly null sequence in the subspace
  $Z_0$, where $Z_0^* = \ell_1(X_n^*)$.

  Let $z_1,\ldots,z_m \in S_Z$.
  Choose $u = (u_n)_{n=1}^\infty \in S_Z$ such that
  $\|u_n\| = 1$ for all $n$.
  Define
  \begin{equation*}
    y_k^i := z_i - P_k(z_i)
    \quad \text{and} \quad
    w_k := P_k(u).
  \end{equation*}
  Then $y_k^i \to_k z_i$ weakly and $w_k \to_k 0$ weakly
  since both $(P_k(z_i))_{k=1}^\infty$ and $(P_k(u))_{k=1}^\infty$ are weakly null.
  By definition $\|y_k^i \pm w_k\| = 1$.
  From Theorem~\ref{ssd2p:equiv_forms}~\ref{item:ssd2p-char-d}
  we see that $Z$ has the SSD2P.
\end{proof}


\section{Subspaces with the SSD2P}\label{sec: Subspaces}

We show that the SSD2P behaves similarly to the SD2P by passing to subspaces. 


Let $X$ be a Banach space and $Y$ a subspace of $X$. Following \cite{ALN2} we say that $Y$
is an \emph{almost isometric ideal} (ai-ideal)
in $X$ if for every finite-dimensional subspace
$E$ of $X$ and every $\varepsilon > 0$ there
exists a bounded linear operator $T: E \to Y$
such that
$(1-\varepsilon)\|e\| \le \|Te\| \le (1+\varepsilon)\|e\|$
and $Te = e$ for all $e \in E \cap Y$.

Recall that $\varphi : Y^* \to X^*$ is called a \emph{Hahn--Banach extension operator} if $\varphi(y^*)(y) = y^*(y)$ and $\|\varphi(y^*)\|
= \|y^*\|$ for all $y \in Y$ and $y^* \in Y^*$. Theorem~1.4 in \cite{ALN2} says that:
 \emph{If $Y$ is an ai-ideal in $X$, then there exists a Hahn--Banach extension operator $\varphi : Y^* \to X^*$ such that for every $\varepsilon>0$, every finite-dimensional subspace $E\subset X$ and every finite-dimensional subspace $F\subset Y^{\ast}$ there exists $T\colon E\rightarrow Y$ which satisfies:}
\begin{itemize}
\item[(a)] \emph{$Te=e$ for all $e\in E\cap Y.$}
\item[(b)] \emph{$(1 - \varepsilon) \|e\| \le \|Te\|
    \le (1 + \varepsilon) \|e\|$
    for all $e \in E$.}

\item[(c)] \emph{$\varphi(f)(e) = f(Te)$ for all $e \in E$
    and $f \in F$.}

\end{itemize}
The Principle of Local Reflexivity says that every Banach space is an ai-ideal in its bidual.

\begin{proposition}\label{prop:ai-inherit-ssd2p}
  Let $X$ be Banach space and $Y$ be its closed subspace. If $X$ has the SSD2P and $Y$ is an ai-ideal in $X$, then $Y$ has the SSD2P.
\end{proposition}

\begin{proof}
 Let $\varphi : Y^* \to X^*$ be a Hahn--Banach extension operator
  connected to the local projections.
  Let $y_1,\ldots,y_n \in S_Y$,
  $U_i \in \mathcal{O}(y_i)$ and $V \in \mathcal{O}(0)$
  in $Y$, and $\varepsilon > 0$.

  Let $\delta > 0$ be so small that
  $(1+\delta)^2 + \delta < 1 + \varepsilon$.
  By choosing $\delta$ even smaller if necessary
  there exist finite sets
  $A_i \subset S_{Y^*}$ and $B \subset S_{Y^*}$
  such that
  \begin{equation*}
    U_i \supset \bar{U}_i := \{
    y \in B_Y : |y^*(y - y_i)| < \delta, y^* \in A_i
    \}
  \end{equation*}
  and
  \begin{equation*}
    V \supset \bar{V} := \{
    y \in B_Y : |y^*(y)| < \delta, y^* \in B
    \}.
  \end{equation*}
  Define corresponding neighborhoods in $X$ by
  \begin{equation*}
    \tilde{U}_i := \{
    x \in B_X : |\varphi(y^*)(x - y_i)| < \frac{\delta}{2}, y^* \in A_i
    \}
  \end{equation*}
  and
  \begin{equation*}
    \tilde{V} := \{
    x \in B_X : |\varphi(y^*)(x)| < \frac{\delta}{2}, y^* \in B
    \}.
  \end{equation*}
  By Theorem~\ref{ssd2p:equiv_forms}~\ref{item:ssd2p-char-e}, there exist $x_i \in \tilde{U}_i \cap S_X$
  and $z \in \tilde{V} \cap S_X$ such that
  $\|x_i \pm z\| \le 1 + \delta$.

  Define $E := \linspan\{x_1,\dots,x_n,y_1,\dots,y_n,z\} \subset X$ and
  $F := \linspan(A_1\cup\ldots\cup A_n\cup B) \subset Y^*$.
  Both $E$ and $F$ are finite-dimensional.
  Since $Y$ is an ai-ideal in $X$ there exists
  a bounded linear operator $T : E \to Y$ such that
  \begin{enumerate}
  \item
    $Te = e$ for all $e \in E \cap Y$.
  \item
    $(1 - \frac{\delta}{2}) \|e\| \le \|Te\|
    \le (1 + \frac{\delta}{2}) \|e\|$
    for all $e \in E$.
  \item
    $\varphi(f)(e) = f(Te)$ for all $e \in E$
    and $f \in F$.
  \end{enumerate}
  Define $u_i := Tx_i/\|Tx_i\|$
  and $v := Tz/\|Tz\|$.
  Then $\|u_i - Tx_i\| \le \delta/2$
  and $\|v - Tz\| \le \delta/2$
  hence
  \begin{align*}
    |y^*(u_i - y_i)|
    &\le |y^*(Tx_i-y_i)| + \frac{\delta}{2}
    = |y^*(T(x_i - y_i))| + \frac{\delta}{2} \\
    &= |\varphi(y^*)(x_i - y_i)| + \frac{\delta}{2}
      < \frac{\delta}{2} + \frac{\delta}{2} = \delta
  \end{align*}
  and similarly $|y^*(v)| < \delta$.
  This means that $u_i \in U_i \cap S_Y$,
  for all $i \in \setN$,
  and $v \in V \cap S_Y$.
  Finally,
  \begin{align*}
    \| u_i \pm v \|
    &= \|Tx_i - u_i\| + \|T(x_i \pm z)\|
    + \|Tz - w\| \\
    &\le \frac{\delta}{2}
    + (1 + \frac{\delta}{2})\|x_i \pm z\|
    + \frac{\delta}{2}
      \le (1+\delta)^2 + \delta
      \le 1 + \varepsilon.
  \end{align*}
  From Theorem~\ref{ssd2p:equiv_forms}~\ref{item:ssd2p-char-e}
  we get that $Y$ has the SSD2P.
\end{proof}

\begin{corollary}\label{cor:ssd2p-inherit-from-bidual}
  If $X$ is a Banach space such that $X^{**}$
  has the SSD2P, then $X$ has the SSD2P.
\end{corollary}




Recall that a Banach space $X$ is \emph{strongly regular} if every closed and convex subset of $B_X$ has convex combinations of slices with arbitrarily small diameter. For a deeper discussion of strong regularity and related concepts we refer the reader to \cite{GGMS}.
\begin{proposition}\label{prop: strongly regular}
  Let $X$ be a Banach space and $Y$ a closed subspace. If $X$ has the
  SSD2P and $X/Y$ is strongly regular, then $Y$ has the SSD2P too. In
  particular, SSD2P passes down to finite co-dimensional subspaces.
\end{proposition}

We omit the proof of Proposition~\ref{prop: strongly regular}, because using Theorem~\ref{ssd2p:equiv_forms}~\ref{item:ssd2p-char-b} it is similar to the proof of \cite[Theorem~2.2~(iii)]{becerra_guerrero_2016}.

Recall that a subspace $Y$ of a Banach space $X$ is an \emph{M-ideal} in $X$
if there exists a bounded linear projection $P : X^* \to X^*$ such
that $\ker P = Y^\perp$ and
\begin{equation*}
  \|x^*\| = \|Px^*\| + \|x^* - Px^*\|
\end{equation*}
for all $x^* \in X^*$.

\begin{proposition}\label{prop: M-ideal}
  Let $Y$ be a proper closed subspace of a Banach space $X$.
  If $Y$ is an M-ideal in $X$ and $Y$ has the SSD2P,
  then $X$ has the SSD2P.
\end{proposition}

\begin{proof} The proof is modelled on the proof of \cite[Proposition~3]{HL1}.
  Let $S_i := S(B_{X},x_i^*,\alpha_i)$, $i\in\{1,\dots,n\}$, be slices and let $\varepsilon > 0$.

  Let $P:X^* \to X^*$ with $\ker P = Y^\perp$ be the
  M-ideal projection.
  Define
  \begin{equation*}
    y_i^* := \frac{Px_i^*}{\|Px_i^*\|}
    \quad \mbox{and} \quad
    \beta_i :=
    \frac{\varepsilon(1-\|Px_i^*\|)+\varepsilon^2}{\|Px_i^*\|} > 0.
  \end{equation*}
  Since $Y$ has the SSD2P there exist
  $u_i \in S(B_Y,y_i^*,\beta_i)$ and $v \in B_Y$ with
  $u_i \pm v \in
  S(B_Y,y_i^*,\beta_i)$ and $\|v\| > 1-\varepsilon$.
  Note that we then have $|y_i^*(v)| < \beta_i$.
  The choice of $\beta_i$ means that
  \begin{equation*}
    Px_i^*(u_i) > (\|Px_i^*\| - \varepsilon)(1+\varepsilon).
  \end{equation*}
  If we happen to have $Px_i^* = 0$ we just set $u_i = 0$
  and use the $v$ we get from the rest of the slices.
  And if $Px_i^*=0$ for all $i$ then use any $v \in S_Y$.

  Find $x_1,\ldots,x_n$ such that
  \begin{equation*}
    (x_i^*-Px_i^*)(x_i) > 
    (\|x_i^*-Px_i^*\| - \varepsilon)(1 + \varepsilon).
  \end{equation*}
  By Proposition~2.3 in \cite{Dirk6} for each $i$ there is
  a net $z_{\alpha,i}$ in $Y$ such that $z_{\alpha,i} \to x_i$
  in the $\sigma(X,Y^*)$-topology and
  \begin{equation*}
    \limsup \|y + (x_i-z_{\alpha,i})\| \le 1
  \end{equation*}
  for all $y \in B_Y$.
  Hence we may choose $z_i \in Y$ such that
  \begin{gather*}
    \|u_i + x_i - z_i\| < 1+\varepsilon \\
    \|u_i \pm v + x_i - z_i\| < 1+\varepsilon \\
    |P(x_i^*)(x_i-z_i)| < \varepsilon.
  \end{gather*}
  Define
  \begin{equation*}
    y_i := \frac{u_i + x_i - z_i}{1+\varepsilon}
    \quad \mbox{and} \quad
    w := \frac{v}{1+\varepsilon}.
  \end{equation*}
  Then
  \begin{align*}
    x_i^*(y_i) &
    = \frac{x_i^*(u_i + x_i - z_i)}{1+\varepsilon} \\
    &= \frac{Px_i^*(u_i) + (x_i^*-Px_i^*)(x_i) + Px_i^*(x_i-z_i)}
    {1+\varepsilon} \\
    &> \frac{(\|Px_i^*\|-\varepsilon)(1+\varepsilon) +
      (\|x_i^*-Px_i^*\|-\varepsilon)(1+\varepsilon) - \varepsilon}
    {1+\varepsilon} \\
    &> \|x_i^*\| - 3\varepsilon=1-3\varepsilon. 
  \end{align*}
  Also $1 \ge \|w\| \ge
  \frac{1-\varepsilon}{1+\varepsilon}$
  and
  \begin{align*}
    x_i^*(y_i \pm w)
    &> 1 - 3\varepsilon
    \pm \frac{\|Px_i^*\|}{1+\varepsilon}y_i^*(v) >1-3\varepsilon-\frac{\|P\xast_i\|}{1+\varepsilon}\beta_i\\
    &= 1 - 3\varepsilon
    - \frac{\varepsilon - \varepsilon\|Px_i^*\| + \varepsilon^2}
    {1+\varepsilon}>1-4\varepsilon.
  \end{align*}
  Since $\varepsilon > 0$ is arbitrary
  we can choose it as small as we like so that
  $y_i \in S_i$,
  $y_i \pm w \in S_i$
  and $\|w\|$ is as close to $1$
  as we like.
\end{proof}

The SD2P-version of the following result is {\cite[Theorem 4.10]{ALN}}, its proof in \cite{ALN} actually proves the SSD2P-version.

\begin{theorem}\label{thm: X and X** SSD2P}
Let $X$ be a Banach space and $Y$ its proper closed subspace. If $Y$ is an M-ideal in $X$, that is
  $\Xast = Z \oplus_1 Y^{\perp}$
  for some nonempty subspace $Z$ of $\Xast$, and moreover, if $Z$ is $1$-norming for $X$,
  then both $X$ and $Y$ have the SSD2P.

  In particular, if $X$ is non-reflexive and
  an M-ideal in $X^{**}$, then both $X$
  and $X^{**}$ have the SSD2P.
\end{theorem}

\begin{remark}
  Similar results to Proposition~\ref{prop: M-ideal} and Theorem~\ref{thm: X and X** SSD2P} \emph{cannot} hold for ASQ spaces,
  because $c_0$ is an M-ideal in $\ell_\infty=(c_0)^{\ast\ast}$ and $c_0$ is ASQ, but $\ell_\infty$ is not ASQ.
\end{remark}

\section{Lipschitz spaces with the w$^\ast$-SSD2P}
\label{sec: Lipschitz}

Recall from \cite{PR_CharacterizationOH} that a metric space $M$ has the \emph{long trapezoid property} (LTP) if for every finite subset $N\subset M$ and $\varepsilon>0$, there exist $u,v\in M, u\neq v$, such that
\[
(1-\varepsilon)(d(x,y)+d(u,v))\leq d(x,u)+d(y,v)
\]
holds for all $x,y\in N$. In \cite[Theorem~3.1]{PR_CharacterizationOH} the authors prove that $M$ has the LTP if and only if $\Lip_0(M)$ has the w$^\ast$-SD2P, that is, every finite convex combination
of weak$^\ast$ slices of $B_{\Lip_0(M)}$ has diameter two. We show that for some $M$ with the LTP the space Lip$_0(M)$ even has the weak$^\ast$ version of the SSD2P (see Theorem~\ref{thm: Lip(M) has w*-SSD2P} below).


\begin{definition}
  A dual Banach space $X^\ast$ has the
  \emph{weak$^\ast$ symmetric strong diameter 2 property}
  (w$^\ast$-SSD2P) if for every finite
  family $\{S_i\}^n_{i=1}$ of weak$^\ast$ slices of $B_{X^\ast}$ and
  $\varepsilon>0$ there exist $x^\ast_i\in S_i$ and
  $\yast \in B_{X^\ast}$, independent of $i$,
  such that $x^\ast_i \pm \yast  \in S_i$ for every
  $i \in \setN$ and
  $\n{\yast }>1-\varepsilon$.
\end{definition}

In a dual space the (S)SD2P clearly implies the w$^\ast$-(S)SD2P. The space $(C[0,1])^\ast$ has the w$^\ast$-SD2P and fails the SD2P (see \cite[Example~1.1]{haller_duality_2015}).  We do not know whether the w$^\ast$-SSD2P and the SSD2P for a dual space are really different. However, the w$^\ast$-SSD2P is stronger than the w$^\ast$-SD2P. Indeed, $\ell_\infty\oplus_1 \ell_\infty$ has the SD2P (see \cite{ALN}, hence also the w$^\ast$-SD2P), but $\ell_1$-sums never have the w$^\ast$-SSD2P (the proof is similar to the one of Theorem~\ref{thm: direct sums with SSD2P}). We also note that a Banach space $X$ has the SSD2P if and only if $\Xastast$ has the w$^\ast$-SSD2P, because by Goldstine's theorem $B_X$ is w$^\ast$-dense in $B_{\Xastast}$ and the norm on $\Xastast$ is w$^\ast$-lower semicontinuous.


Let $M$ be a pointed metric space with metric
$d$ and a base point denoted by $0$.
The space $\Lip_0(M)$ of all Lipschitz functions $f\colon M\rightarrow \R$ with $f(0) = 0$ is a Banach space with norm
\[
 \|f\|=\sup\left\{\frac{|f(x) - f(y)|}{d(x,y)}\colon x,y\in M, x\neq y\right\}.
\]
It is known that $\Lip_0(M)$ is a dual space, whose canonical predual is Lipschitz-free space $\mathcal{F}(M)$, the norm closed linear subspace of $\Lip_0(M)^\ast$ spanned by the evaluation functionals $\delta_x$ with $x\in M$. If $\mu=\sum_{i=1}^n a_i\delta_{x_i}$ is an element in $\mathcal{F}(M)$ with $x_i\in M\setminus\{0\}$ and $a_i\neq 0$ for every $i\in\{1,\dots,n\}$, then we will denote the \emph{support} of $\mu$ by $\supp(\mu):=\{x_1,\dots,x_n\}$. 



\begin{proposition}\label{prop: M unbounded}
  If $M$ is an unbounded metric space,
  then $\Lip_0(M)$ has the w$^\ast$-SSD2P.
\end{proposition}

\begin{proof}
  Let $n\in\N$,
  $S_1 = S(B_{\Lip_0(M)},\mu_1,\alpha_1),\dots, S_n = S(B_{\Lip_0(M)},\mu_n,\alpha_n)$
  be weak$^\ast$ slices of $B_{\Lip_0(M)}$, where $\mu_i\in \spann\{\delta_x\colon x\in M\}$, and $\varepsilon>0$.
  We want to show that there exist $f_i\in S_i$
  and $\varphi\in B_{\Lip_0(M)}$ such that
  \begin{equation*}
    f_i \pm \varphi \in S_i
    \qquad \text{ and } \qquad \|\varphi\| > 1-\varepsilon.
  \end{equation*}
  Choose $g_i \in S_i$ with $g_i(\mu_i) = 1$
  for $i \in \setN$. Denote
  by $N := \{0\}\cup \bigcup_{i=1}^n \supp(\mu_i)$. The main idea of the proof is to find norm preserving extensions $f_i$ of $g_i|_N$ such that $f_i|_{M\setminus B(0,s)}=0$ and $\varphi|_{B(0,t)}=0$ for suitable $0<s<t$.
  Since $N$ is a finite subset of $M$, there is an $r > 0$ such that
  $N \subset B(0,r)$.
  Let $s := 2r$.
  Then for every $x\in B(0,r)$ and
  $y \in M\setminus B(0,s)$
  we have $d(x,0) \leq d(x,y)$.
  Since $M$ is unbounded there exists $u \in M \setminus B(0,s)$.

  Let $\delta > 0$ be such that
  $(1-\delta)^2 > \max\{1-\varepsilon, 1-\alpha_i\}$.
  Find a $t > 0$ such that for every $x\in B(0,s)$ and
  $y \in M \setminus B(0,t)$ one has
  \begin{equation*}
    d(x,y)\geq (1-\delta)(d(x,u)+d(u,y)).
  \end{equation*}
  For example, any $t$ with $\delta t \geq 2(s + d(0,u))$ does the job.

  Since $M$ is unbounded there exists $v \in M \setminus B(0,t)$ such that
  \begin{equation*}
    \frac{d(v,0) - t}{d(v,0)} > 1 - \delta, \text{ that is, } \delta\cdot d(v,0)>t.
  \end{equation*}
  Define $\tilde{\varphi} \colon B(0,t) \cup \{v\} \to \R$
  by $\tilde{\varphi}|_{B(0,t)} = 0$ and $\tilde{\varphi}(v) = d(v,0) - t$.
  Then $\|\tilde{\varphi}\| \leq 1$,
  because for any $x\in B(0,t)$ we have
  \begin{equation*}
  \frac{|\tilde{\varphi}(v) - \tilde{\varphi}(x)|}{d(v,x)}
  =
    \frac{|\tilde{\varphi}(v) - 0|}{d(v,x)}
    \leq
    \frac{d(v,0) - t}{d(v,0) - d(0,x)}
    \leq 1.
  \end{equation*}
Also $\|\tilde{\varphi}\| > 1 - \delta$, because
  \begin{equation*}
    \|\tilde{\varphi}\| \geq
    \frac{\tilde{\varphi}(v) - \tilde{\varphi}(0)}{d(v,0)}
    = \frac{d(v,0) - t}{d(v,0)}
    > 1 - \delta.
  \end{equation*}
 For every $i\in\{1,\dots,n\}$ define $\tilde{f_i} \colon N \cup (M \setminus B(0,s)) \to \R$ by
  $\tilde{f_i}|_N = g_i$ and $f_i|_{M \setminus B(0,s)} = 0$.
  Then $\|\tilde{f_i}\| \leq 1$,
  because for any $x \in N$ and $y \in M\setminus B(0,s)$
  we have
  \begin{equation*}
  |\tilde{f_i}(x) - \tilde{f_i}(y)|=|\tilde{f_i}(x)-\tilde{f_i}(0)|\leq d(x,0)\leq d(x,y).
  \end{equation*}
  Consider $f_i :=(1 - \delta)\tilde{f_i}$ and $\varphi :=(1 - \delta)\tilde{\varphi}$
  and extend them norm preservingly to $M$.  Observe that $\|f_i \pm \varphi\| \leq 1$,
  because for any $x \in B(0,s)$ and $y \in M\setminus B(0,t)$
  we have
    \begin{align*}
    |(f_i\pm \varphi)(x)&-(f_i\pm \varphi)(y)|=|f_i(x) \pm \varphi(y)|\\
    &\leq 
    |f_i(x)| +|\varphi(y)|   =    |f_i(x)-f_i(u)|+|\varphi(y)-\varphi(u)|\\
    &\leq 
    (1-\delta)d(x,u) +(1-\delta)d(u,y)\leq d(x,y).
  \end{align*}
  
Finally, note that $\|\varphi\|=(1-\delta)\|\tilde{\varphi}\|>(1-\delta)^2>1-\varepsilon$, and 
\begin{align*}
(f_i\pm \varphi)(\mu_i)&=f_i(\mu_i)=(1-\delta)\tilde{f_i}(\mu_i)\\
&=(1-\delta)g_i(\mu_i)=1-\delta\\
&>(1-\delta)^2>1-\alpha_i.
\end{align*}
\end{proof}

\begin{proposition}\label{prop: Lip and discrete}
  If $M$ is an infinite discrete metric space, then
  $\Lip_0(M)$ has the w$^\ast$-SSD2P.
\end{proposition}

\begin{proof}
  Let $n\in\N$,
  $S_1 = S(B_{\Lip_0(M)},\mu_1,\alpha_1),\dots, S_n = S(B_{\Lip_0(M)},\mu_n,\alpha_n)$
  be weak$^\ast$ slices of $B_{\Lip_0(M)}$, where $\mu_i\in\spanof\{\delta_x\colon x\in M\}$.
  
  Let $N=\{0\}\cup\bigcup_{i=1}^n\supp(\mu_i)$. For every $i\in\{1,\dotsc,n\}$ choose $g_i \in S_{Lip_0(N)}$ such that $g_i(\mu_i)=1$ and let $x_i,y_i\in N$ be such that \[g_i(x_i)-g_i(y_i)=d(x_i,y_i)=1.\]
  Fix any two different elements $u,v\in M\setminus N$. Define $f_i\in S_i$ and $\varphi\in S_{Lip_0(M)}$ by setting
  \[
  f_i(x):=\begin{cases}g_i(x),& \text{if $x\in N$},\\\dfrac{g_i(x_i)+g_i(y_i)}{2},& \text{if $x=u$ or $x=v$},\\ 0& \text{elsewhere},
  \end{cases}
  \]
  and
  \[
  \varphi(x):=\begin{cases}\dfrac{g_i(x_i)-g_i(y_i)}{2}=\dfrac12,& \text{if $x=u$},\\\dfrac{g_i(y_i)-g_i(x_i)}{2}=-\dfrac12,& \text{if $x=v$},\\ 0& \text{elsewhere}.
  \end{cases}
  \]
  Then $f_i, f_i\pm\varphi\in S_i$, and $\|\varphi\|=1$.

\end{proof}

For $n\in \N$ denote by
\[
K_n:=\{x \in \ell_\infty \colon x(k)\in\{0,1,\dots,n\}\}
\]
with metric inherited from $\ell_\infty$. Note that $\Lip_0(K_n)$ has the w$^\ast$-SD2P, because $K_n$ has the LTP (see Lemma~\ref{lemma: K_2}). Also, observe that Proposition~\ref{prop: Lip and discrete} shows that $\Lip_0(K_1)$ has the w$^\ast$-SSD2P. In Proposition~\ref{prop: Lip(K_2) has w*-SSD2P} below, we prove that $\Lip_0(K_2)$ has the w$^\ast$-SSD2P. However, we do not know whether $\Lip_0(K_n)$ has the w$^\ast$-SSD2P for every $n\in\N$. It is also unknown, whether every slice of the unit ball of $\Lip_0(K_n)$ has diameter two for every $n\in \N$ (see \cite[p.~114]{Ivakhno}). 

\begin{lemma}\label{lemma: K_2}
Let $n\in\N$ and let $N$ be a finite subset of $K_n$. Then there are $u,v\in K_n\setminus N$ satisfying
\begin{enumerate}
    \item[(a)] $d(u,v)=1$;
    \item[(b)] For all $x\in N$ one has $d(x,u)=d(x,v)$;
    \item[(c)] For all $x,y\in N$ one has $d(x,y)\leq d(x,u)$ or $d(x,y)\leq d(y,u)$.
\end{enumerate}
In particular, $K_n$ has the LTP.
\end{lemma}
\begin{proof}
Choose an $u\in K_n\setminus N$ such that $u(i)\in \{0,n\}$ for every $i\in\N$. For such an element $u$ the condition (c) holds.  We will construct a suitable $v\in K_n\setminus N$, such that it differs from $u$ in only one coordinate $i_0$. Let $I\subset \N$ be a finite subset such that for every $x\in N$ there is an $i\in I$ such that $d(x,u)=|x(i)-u(i)|$. Fix an $i_0\in \N\setminus I$. Let $v(i_0)$ be such that $|u(i_0)-v(i_0)|=1$. Hence condition (a) holds and we will check condition (b). Let $x\in N$. Clearly, $d(x,v)\geq d(x,u)$, because $u(i)=v(i)$ for all $i\in I$. For the reverse inequality, observe that
\[
  |x(i_0)-v(i_0)|=\begin{cases}1,& \text{if $x(i_0)=u(i_0)$},\\|x(i_0)-u(i_0)|-1,& \text{if $x(i_0)\neq u(i_0)$}.
  \end{cases}
  \]
 Hence, $d(x,u)\geq d(x,v)$, and condition (b) holds.
\end{proof}

\begin{proposition}\label{prop: Lip(K_2) has w*-SSD2P}
The Banach space $Lip_0(K_2)$ has the w$^\ast$-SSD2P.
\end{proposition}





\begin{proof}
Consider $n\in\mathbb N$ and weak$^\ast$ slices $S_1=S(B_{\Lip_0(K_2)},\mu_1,\alpha_1),\dotsc,S_n=S(B_{\Lip_0(K_2)},\mu_n,\alpha_n)$ of $B_{Lip_0(K_2)}$, where $\mu_i\in \operatorname{span}\{\delta_x\colon x\in M\}$. We show that there exist $f_i\in S_i$
  and $\varphi\in B_{\Lip_0(K_2)}$ such that
  \begin{equation*}
    f_i \pm \varphi \in S_i
    \qquad \text{ and } \qquad \|\varphi\|=1.
  \end{equation*}

Let $N=\{0\}\cup\bigcup_{i=1}^n\operatorname{supp}(\mu_i)$ and let $u,v\in K_2$ be as in Lemma~\ref{lemma: K_2} for $N$. For every $\mu_i$ choose $g_i\in S_{Lip_0(K_2)}$ such that $g_i(\mu_i)=1$. The main idea of the proof is to define $\varphi$ such that $\varphi=0$ outside $\{u,v\}$ and $f_i$ are norm preserving extensions of $g_i|_N$ satisfying $f_i(u)=f_i(v)$ and $|f_i(x)-f_i(y)|\leq 1$ for every $x,y\notin N$.
For $j\in \{1,2\}$ set
\[N_j:=\{x\in N\colon d(x,u)=d(x,v)=j\}.\]
For every $g_i$ define its norm preserving extension $g^{+}_i$ from $N$ to $N\cup\{u,v\}$ by taking
\begin{align*}
 g^{+}_i(u)&:=\min\{g_i(x) +d(x,u)\colon x\in N\},\\
 g^{+}_i(v)&:=\max\{g^{+}_i(x) -d(x,v)\colon x\in N\cup\{u\}\}.
\end{align*}
This means that $g^{+}_i$ is the maximal extension from $N$ to $u$ and then minimal extension to $v$ preserving the Lipschitz constant (see \cite{MS} or \cite[p.~18]{W}).
Note that $g^{+}_i(v)+1=g^{+}_i(u)$. Indeed, for every $k,l\in\{1,2\}$,
\[
(\min_{x\in N_k}g_i(x)+k)-(\max_{x\in N_l}g_i(x)-l)\geq 1,
\] that is, 
\[
\max_{x\in N_l}g_i(x)-\min_{x\in N_k}g_i(x)\leq k+l-1,
\]
because for $x\in N_l$ and $y\in N_k$
\[
g_i(x)-g_i(y)\leq d(x,y)\leq \begin{cases}
d(x,u)=l\leq k+l-1\\
\text{or}\\
d(y,u)=k\leq k+l-1,
\end{cases}
\]
by Lemma~\ref{lemma: K_2}.

If there is an element $x\in K_2\setminus N$ such that 
\[
N^1_2(x):=\{y\in N_2\colon d(x,y)=1\}\neq\emptyset,
\]
then choose arbitrarily $a^x_i$ from the set
\[
 [\max_{y\in N^1_2(x)}g_i(y)-1,\min_{y\in N^1_2(x)}g_i(y)+1 ]\cap [g^{+}_i(u)-1, g^{+}_i(u)].
\]
Note that the latter intersection is nonempty, because \[
\max_{y\in N^1_2(x)}g_i(y)-1\leq g^{+}_i(u)=g^{+}_i(v)+1
\]
and 
\[
\min_{y\in N^1_2(x)}g_i(y)+1\geq g^{+}_i(u)-1.
\]

Define 
\[
\varphi(x):=
\begin{cases}
\frac12,\quad \text{if } x=u,\\
-\frac12,\quad \text{if } x=v,\\
0\quad \text{elsewhere},
\end{cases}
\]
and
\[
f_i(x):=
\begin{cases}
g_i(x), \quad \text{if } x\in N,\\
a^x_i,\quad \text{if } x\in K_2\setminus N \text{ and } d(x,N_2)=1,\\
g^{+}_i(u)-\frac12\quad \text{elsewhere}.
\end{cases}
\]
Then $f_i\in S_i$, $(f_i\pm \varphi)(\mu_i)=f_i(\mu_i)>1-\alpha_i$, and $\|\varphi\|=1$. To check that $\|f_i \pm \varphi\|\leq 1$, we argue by cases:
\begin{itemize}
    \item If $x\in N$ and $y=u$, then
    \begin{align*}
        |(f_i&\pm \varphi)(x)-(f_i\pm \varphi)(u)|=\left |g_i(x)-\left(g^{+}_i(u)-\frac12\pm \frac12\right) \right|=\\
        &=\begin{cases}
|g_i(x)-g^{+}_i(u)|\leq d(x,u)\\
\text{or}\\
|g_i(x)-g^{+}_i(u)+1|=|g_i(x)-g^{+}_i(v)|\leq d(x,v)=d(x,u).
\end{cases}
    \end{align*}
\item If $x\in K_2\setminus(N\cup\{u,v\})$ and $y=u$, then
    \begin{align*}
        |(f_i&\pm \varphi)(x)-(f_i\pm \varphi)(u)|=\\
        &=\begin{cases}
|a^x_i-(g^{+}_i(u)-\frac12\pm \frac12)|\leq 1\leq d(x,u)\\
\text{or}\\
|(g^{+}_i(u)-\frac12)-(g^{+}_i(u)-\frac12\pm \frac12)|=\frac12 \leq d(x,u),
\end{cases}
    \end{align*}
    because $a^x_i\in [g^{+}_i(u)-1, g^{+}_i(u)]$.
    
\item If $x\in K_2\setminus(N\cup\{u,v\})$ and $y\in N$, then
    \begin{align*}
        |(f_i\pm \varphi)(x)-(f_i\pm \varphi)(y)|=|f_i(x)-f_i(y)|\leq d(x,y),
    \end{align*}
    because $\|f_i\|\leq 1$.
\item The other cases are trivial or similar to the ones above.
\end{itemize}
Hence, $\|f_i\pm \varphi\|\leq 1$, which completes the proof.
\end{proof}

We now collect the known examples of metric spaces $M$ such that $\Lip_0(M)$ has the w$^\ast$-SSD2P.

\begin{theorem}\label{thm: Lip(M) has w*-SSD2P}
  If $M$ is an infinite metric space satisfying at least one of the following conditions:
  \begin{itemize}
      \item[(a)] $\sup\{d(x,y)\colon x,y\in M,x\neq y\}=\infty$;
      \item[(b)] $\inf\{d(x,y)\colon x,y\in M,x\neq y\}=0$;
      \item[(c)] $M$ is a discrete metric space;
      \item[(d)] $M=K_n$, where $n\in\{1,2\}$,
  \end{itemize}
 then $Lip_0(M)$ has the w$^\ast$-SSD2P.
\end{theorem}
\begin{proof}
(a), (c), and (d) are Propositions~\ref{prop: M unbounded}, \ref{prop: Lip and discrete}, and \ref{prop: Lip(K_2) has w*-SSD2P}, respectively. An inspection of the proof of Theorem~2.4 in \cite{BLR_octahedralityLipschitz} shows (b).
\end{proof}

\section{Questions}\label{sec: Questions}

Let us end the paper with some questions that are suggested by the current work:
\begin{ques}
  If a Banach space has the SSD2P must it then contain an isomorphic copy of $c_0$?
\end{ques}

\begin{ques}
  Does there exist a dual Banach space with the w$^\ast$-SSD2P and without the SSD2P?
\end{ques}


\begin{ques}
  If $M$ has the LTP, does then $\Lip_0(M)$ have the w$^\ast$-SSD2P?
\end{ques}

\section*{Acknowledgements}
The authors wish to express their thanks to Indrek Zolk for his collaboration in proving Lemma~\ref{lemma: K_2}.

\bibliographystyle{amsplain}

\end{document}